\documentclass[10pt]{article}
\usepackage{amssymb}
\usepackage{graphicx}
\usepackage{xcolor} 
\usepackage{tensor}
\usepackage{fullpage} 
\usepackage{amsmath}
\usepackage{amsthm}
\usepackage{verbatim}
\usepackage{txfonts,bm,mathtools}

\usepackage{enumitem}
\setlist[enumerate]{leftmargin=1.5em}
\setlist[itemize]{leftmargin=1.5em}

\usepackage{hyperref}
\usepackage{seqsplit,mathtools} 
\usepackage{caption}
\usepackage{subcaption}

\mathtoolsset{showonlyrefs}

\definecolor{green}{rgb}{0,0.8,0} 

\newtheorem{thm}{Theorem}[section]

\newtheorem{lem}[thm]{Lemma}
\newtheorem{prop}[thm]{Proposition}

\theoremstyle{definition}

\theoremstyle{remark}
\newtheorem{rmk}[thm]{Remark}

\numberwithin{equation}{section}
\newcommand{\nrm}[1]{\Vert#1\Vert}

\newcommand{\tld}[1]{\widetilde{#1}}
\newcommand{\br}[1]{\overline{#1}}

\newcommand{\nnrm}[1]{{\vert\kern-0.25ex\vert\kern-0.25ex\vert #1 
		\vert\kern-0.25ex\vert\kern-0.25ex\vert}}

\newcommand{\supp}{{\mathrm{supp}}\,}

\newcommand{\rd}{\partial}
\newcommand{\nb}{\nabla}

\newcommand{\ift}{\infty}

\newcommand{\alp}{\alpha}
\newcommand{\bt}{\beta}

\newcommand{\dlt}{\delta}

\newcommand{\veps}{\varepsilon}

\newcommand{\lmb}{\lambda}

\newcommand{\tht}{\theta}

\newcommand{\omg}{\omega}

\newcommand{\bbR}{\mathbb R}

\newcommand{\To}{\longrightarrow}

\begin{document}

\bibliographystyle{plain}
\title{Global regularity for some axisymmetric Euler flows in $\mathbb{R}^{d}$}
\author{Kyudong Choi\thanks{Department of Mathematical Sciences, Ulsan National Institute of Science and Technology, 50 UNIST-gil, Eonyang-eup, Ulju-gun, Ulsan 44919, Republic of Korea. Email: kchoi@unist.ac.kr}\and In-Jee Jeong\thanks{Department of Mathematical Sciences and RIM, Seoul National University, 1 Gwanak-ro, Gwanak-gu, Seoul 08826, Republic of Korea. Email: injee$ \_ $j@snu.ac.kr}\and 
	Deokwoo Lim\thanks{Department of Mathematical Sciences, Ulsan National Institute of Science and Technology, 50 UNIST-gil, Eonyang-eup, Ulju-gun, Ulsan 44919, Republic of Korea. Email: dwlim@unist.ac.kr} }

\date\today
\maketitle

\begin{abstract}
	We consider axisymmetric Euler flows without swirl in $\mathbb{R}^{d}$ with $d\ge4$, for which the global regularity of smooth solutions is an open problem. When $d = 4$, we obtain global regularity under the assumption that the initial vorticity satisfies some  decay  at infinity and is vanishing at the axis. Assuming further that the initial vorticity is of one sign guarantees global regularity for $d\le 7$. 
\end{abstract}

\renewcommand{\thefootnote}{\fnsymbol{footnote}}
\footnotetext{\emph{2020 AMS Mathematics Subject Classification:} 76B47, 35Q35 }
\footnotetext{\emph{Key words:} High dimensional Euler; axisymmetric flow; vorticity growth; global regularity.}
\renewcommand{\thefootnote}{\arabic{footnote}}

\section{Introduction}

\subsection{Axisymmetric Euler equations without swirl}

In this paper, we are concerned with the question of global regularity for incompressible Euler flows in $\bbR^{d}$ ($d\ge3$) which are \textit{axisymmetric without swirl}. Under this assumption, the velocity of the incompressible Euler equations  \begin{equation}\label{eq_Euler}
	\left\{
	\begin{aligned}
		\rd_{t}u+(u\cdot\nb)u+\nb p&=0,\\
		\nb\cdot u&=0
	\end{aligned}
	\right.
\end{equation} takes the form $u=u^{r}(r,z)e^{r}+u^{z}(r,z)e^{z}$, where $(r,\tht_{1},\tht_{2},\cdots,\tht_{d-3},\phi,z)$ is the cylindrical coordinate system in $\bbR^{d}$. Introducing the scalar vorticity $ \omg:=\rd_{z}u^{r}-\rd_{r}u^{z} $ reduces \eqref{eq_Euler} to 
\begin{align}
	\rd_{t}\omg+u\cdot\nb\omg&=\frac{(d-2)u^{r}}{r}\omg,\label{eq_vortformNd}\\
\end{align} with \begin{equation}\label{eq_urpsi}
\begin{split}
	u^{r}&=-\frac{1}{r^{d-2}}\rd_{z}\psi, \qquad u^{z}=\frac{1}{r^{d-2}}\rd_{r}\psi,
\end{split}
\end{equation}
where $u\cdot\nb = u^r\rd_r + u^z\rd_z$ and $ \psi $ is the scalar stream function introduced below. Alternatively, using the
\textit{relative} vorticity $ \xi:=r^{-(d-2)}\omg $, \eqref{eq_vortformNd} becomes\begin{equation}\label{eq_transport}
	\rd_{t}\xi+u\cdot\nb\xi=0.
\end{equation}
Since $u$ is divergence free, this shows that the $ L^{p}(\bbR^d) $-norms of $ \xi $ are conserved for any $ p\in[1,\ift] $.

\subsection{Main results}

As we shall discuss in more detail below, it is not at all clear whether smooth solutions to the axisymmetric without swirl equation can blow up in finite time or not, for $d\ge4$. We give two conditions which guarantee global regularity. In the following statements, we assume that the corresponding initial velocity $u_0$ belongs to $H^s(\bbR^d)$ for some $s>1+d/2$, so that there exists a local in time unique $H^s$ solution (\cite{Ka,KL}, also see the recent work \cite{Miller}). This in particular guarantees that $\omg_0\in L^\infty(\bbR^d)$.
\begin{thm}\label{thm:1} 
	Consider the case $d = 4$. Assume that the scalar vorticity $\omg_0$ satisfies $r^{-2}\omg_0 \in (L^1\cap L^\infty)(\bbR^4)$ and $r^2\omg_0 \in L^1(\bbR^4)$. Then, the corresponding smooth solution is global in time and satisfies \begin{equation*}
		\begin{split}
						\|\omg(t)\|_{L^{\ift}(\bbR^{4})} \leq C_{1}e^{C_{0}t},
		\end{split}
	\end{equation*} for some $C_0,C_1>0$ depending on the initial data. 
\end{thm}

\begin{thm}\label{thm:2}
	Consider $d\le 7$, and assume that $\omg_0$ is compactly supported in $\bbR^{d}$, satisfies $r^{-(d-2)}\omg_0 \in L^\infty(\bbR^{d})$, and that $\omg_0$ is either non-negative or non-positive everywhere on $\bbR^{d}$. Then, the corresponding solution is global in time and satisfies
	\begin{equation*}
		\begin{split}
			\|\omg(t)\|_{L^{\ift}(\bbR^{d})}&\leq\begin{cases}
				C_{2}(1+t)^{\frac{4(d-2)}{7-d}},\quad &d=4,5,6,\\
				C_{3}e^{C_{4}t},\quad &d=7,
			\end{cases}
		\end{split}
	\end{equation*} for some $ C_{2}=C_{2}(d), C_{3}, C_{4}>0 $ depending on the initial data.
\end{thm}

\begin{rmk}
	In both statements, smooth and compactly supported initial data do not satisfy the assumption $r^{-(d-2)}\omg_0\in L^\infty$ in general, since $d\ge4$. However, they are satisfied for data which are supported away from the symmetry axis. 
\end{rmk}

\subsection{Discussion}

\textbf{BKM criterion in general dimension and regularity in 3D}. It is well-known that smooth solutions to axisymmetric flows without swirl are global when $d = 3$ (\cite{MB}). To explain this, let us recall the Beale--Kato--Majda criterion (\cite{BKM} for $d = 3$ and \cite{KaPo} for general dimensions, also see the recent work \cite{Miller}): in terms of the velocity, $H^{s}(\bbR^d)$ solutions ($s>d/2+1$) to the Euler equations can become singular at time $T^*$ if and only if \begin{equation}\label{eq:BKM}
	\begin{split}
		\int_0^{T^*} \nrm{ \nb \times u(t,\cdot)}_{L^\infty(\mathbb{R}^d)} dt = +\infty. 
	\end{split}
\end{equation} For axisymmetric flows without swirl, the quantity $\nrm{ \nb \times u(t,\cdot)}_{L^\infty(\mathbb{R}^d)} $ is simply equivalent with $\nrm{ \omg(t,\cdot)}_{L^\infty(\mathbb{R}^d)}$, where $\omg$ is the scalar vorticity introduced above. In turn, using the vorticity equation \eqref{eq_vortformNd}, we see that \begin{equation*}
\begin{split}
	\frac{d}{dt} \nrm{ \omg(t,\cdot)}_{L^\infty(\bbR^{d})} \le (d-2)\nrm{ r^{-1}u^r(t,\cdot)}_{L^\infty(\bbR^{d})} \nrm{ \omg(t,\cdot)}_{L^\infty(\bbR^{d})}, 
\end{split}
\end{equation*} which means that for axisymmetric without swirl flows, $\nrm{r^{-1}u^r}_{L^\infty}$ controls singularity formation. Now, when $d = 3$, one may use the following Danchin's estimate in \cite{Danaxi} (see \cite{Danaxi,Raymond} and the references therein for previous works) \begin{equation*} 
\begin{split}
	\nrm{ r^{-1}u^r(t,\cdot)}_{L^\infty(\bbR^3)} \le C \nrm{ r^{-1}\omg(t,\cdot)}_{L^{3,1}(\bbR^3)} = C \nrm{ r^{-1}\omg_0}_{L^{3,1}(\bbR^3)} ,
\end{split}
\end{equation*} where $L^{p,q}$ denotes the Lorenz space. It is important that the assumption $u_0 \in H^s(\bbR^{3})$ with $s>5/2$ ensures that the right hand side is finite, thereby concluding global regularity. We note that it is still possible to have \textit{infinite} time blow-up of $\nrm{\omg(t,\cdot)}_{L^\infty}$ (and related quantities), see \cite{CJ-axi,ChilGil1,ChilGil2,DE,Tam-axi} for instance. 

\medskip

\noindent \textbf{Singularity formation in 3D}. Rather strikingly, Elgindi has recently shown in \cite{Elgindi-3D} that $C^{1,\alp}(\bbR^3)$ velocity solutions to axisymmetric without swirl flows can develop finite time singularity, for $0<\alp$ sufficiently small.\footnote{Danchin's estimate gives global regularity only for $\alp>1/3.$} This clearly demonstrates that the condition $r^{-1}\omg_0\in L^{3,1}$ is not just a technical requirement. When there is nonzero swirl, a recent preprint of Chen--Hou \cite{ChenHou2} gives singularity formation for smooth data in a 3D cylinder. For previous blow-up results for 3D Euler, see \cite{ChenHou,EJE,EJO}. 

\medskip

\noindent \textbf{Singularity formation in higher dimensions}. While it is expected that Elgindi's work \cite{Elgindi-3D} can be extended to higher dimensional axisymmetric flows, many authors have noted the possibility of singularity formation for \textit{smooth} solutions without swirl (\cite{DE,KhYa,Yang,Miller}); this is the content of \textbf{Open Question 7} in the nice review paper of Drivas--Elgindi \cite{DE}. Khesin--Yang considers certain high-dimensional analogues of vortex filaments in $\bbR^d$ and shows the possibility of finite time singularity precisely when $d\ge5$, although rigorously deriving their evolution equations seems to be a very challenging task. This is formally consistent with our results, since their vorticity satisfies an odd symmetry. The authors of \cite{DE} give examples of singularity formation in ``infinite'' dimensional Euler equations, suggesting the possibility of perturbing it to get blow up in sufficiently large dimensions. See also the recent paper of Miller \cite{Miller}, who obtains a singularity formation for certain ``infinite dimensional'' axisymmetric model equation. Our results show that, at least when the dimension is relatively small, some assumptions on the vorticity can be given to ensure global regularity.

\medskip
 
\noindent \textbf{Local vs. global in 4D}.  
While global regularity of axisymmetric solutions is known in 3D  (e.g. see \cite{MB}), the same problem in 4D is open. Miller in \cite[(1.22)]{Miller} obtained the following estimate:
For $S(t):=\sup_{\tau \in [0,t]} \sup_{r>0}\{(r,z)\in \supp\, \omega(\tau)\}$
$$
\dot S(t)\leq C_{\omega_0} S(t)^2,
$$ which gives only local existence. We improve the estimate into (see Proposition \ref{prop_Rtgrowth})
$$
\dot S(t)\leq C_{\omega_0} S(t),
$$ using the following extension of Feng--Sverak \cite{FeSv} type estimate in $\bbR^3$ \eqref{eq_urest}:
$$		\|u^{r}\|_{L^{\ift}(\bbR^{4})}\lesssim\|r^{2}\omg\|_{L^{1}(\bbR^{4})}^{\frac{1}{4}}\bigg\|\frac{\omg}{r^{2}}\bigg\|_{L^{1}(\bbR^{4})}^{\frac{1}{4}}\bigg\|\frac{\omg}{r^{2}}\bigg\|_{L^{\ift}(\bbR^{4})}^{\frac{1}{2}}.$$ This shows that the \textit{only way} for smooth data to blow up in 4D is to have $\omg_0(r) \sim r$ near the axis. 

\medskip
 
\noindent \textbf{One-signed vortex dynamics: growth of the support}. When the vorticity is single signed, one can improve the support growth estimate, thanks to the conservation of angular impulse, which is \begin{equation*}
	\begin{split}
		\int_{\bbR^d} r\omg dx \simeq_d \int_{-\infty}^{\infty} \int_0^\infty r^{d-1}\omg(r,z)drdz. 
	\end{split}
\end{equation*} While this is conserved in general, it is equivalent with the norm $\nrm{r\omg}_{L^1(\bbR^d)}$ only when $\omg$ is single signed, which can be applied to improve the Feng--Sverak type estimate. 
This is why we can obtain global regularity up to $d\le 7$ in Theorem \ref{thm:2}. One may try to improve the condition $d\leq 7$ by studying a confinement argument of \cite{ISG99} from $\bbR^2$. When considering one-signed compactly supported relative vorticity in 3D axi-symmetric setting, 
 vorticities are  confined in the region 
$\{r\leq C t^{1/4}\log t\}$ (\cite{Maffei2001}). Similarly, it is possible to make an estimate for higher dimensional cases, which might bring us beyond the restriction $d\leq7$. The key idea of   confinement  for {one-signed} vorticities came from  the two-dimensional case $\mathbb{R}^2$  by \cite{M94}, \cite{ISG99}, \cite{Serfati_pre} 
  (see also \cite{ILL2003}, \cite{ILL2007}, \cite{CD2019} for other two-dimensional domains).

\section{Proofs}

\subsection{Stream function}

We take $\Pi := \left\{ (r,z) : r\ge0, z \in \bbR \right\}$.
The scalar stream function $\psi$ is defined by \begin{equation}\label{eq_kernelofpsi}
	\begin{split}
		\psi(r,z) = \iint_{\Pi} G_{d}(r,\br{r},z,\br{z}) \omg(\br{r},\br{z}) d\bar{z} d\bar{r} \quad \mbox{with} \quad 	G_{d}(r,\br{r},z,\br{z}) := \frac{2\pi\bt_{d}}{d(d-2)\alp_{d}}\int_{0}^{\pi}\frac{(r\br{r})^{d-2}\cos\tht_{1}\sin^{d-3}\tht_{1}}{[r^{2}+\br{r}^{2}-2r\br{r}\cos\tht_{1}+(z-\br{z})^{2}]^{\frac{d}{2}-1}}d\tht_{1}
	\end{split}
\end{equation}  for $d\ge4$, where  $\alp_{d}$ is the volume of the unit ball in $\bbR^d$ and	\begin{equation*}
\begin{split} 
	\bt_{d}&=\begin{cases}
		1 &\;\text{if}\; d=4,\\
		\bigg(\int_{0}^{\pi}\sin^{d-4}\tht_{2}d\tht_{2}\bigg)\bigg(\int_{0}^{\pi}\sin^{d-5}\tht_{3}d\tht_{3}\bigg)\cdots\bigg(\int_{0}^{\pi}\sin\tht_{d-3}d\tht_{d-3}\bigg) & \;\text{if}\; d\geq5.
	\end{cases}
\end{split}
\end{equation*}


\subsection{Key Estimate for $ u^{r} $}

From \eqref{eq_urpsi}, we have
\begin{equation}\label{eq_urform}
	\begin{split}
		u^{r}(r,z)&=-\frac{1}{r^{d-2}}\int_{0}^{\ift}\int_{-\ift}^{\ift}\rd_{z}G_{d}(r,\br{r},z,\br{z})\omg(\br{r},\br{z})d\br{z}d\br{r} =c_{d}\iint_{\Pi}\frac{\br{r}^{\frac{d}{2}-2}(z-\br{z})}{r^{\frac{d}{2}}}F_{d}'\bigg(\frac{(r-\br{r})^{2}+(z-\br{z})^{2}}{r\br{r}}\bigg)\omg(\br{r},\br{z})d\br{z}d\br{r},
	\end{split}
\end{equation} where $ c_{d}=\frac{2\pi\bt_{d}}{d(d-2)\alp_{d}} $ and
\begin{equation}\label{eq_Fd}
	F_{d}(s):=\int_{0}^{\pi}\frac{\cos\tht\sin^{d-3}\tht}{[2(1-\cos\tht)+s]^{\frac{d}{2}-1}}d\tht,\quad s>0.
\end{equation}

The following is an extension of Feng--Sverak type estimate \cite[Prop. 2.11]{FeSv} to $ \bbR^{d} $.

\begin{prop}\label{prop_urest} 
	We have the estimate 
	\begin{equation}\label{eq_urest}
		\|u^{r}\|_{L^{\ift}(\bbR^{d})}\lesssim\|r^{d-2}\omg\|_{L^{1}(\bbR^{d})}^{\frac{1}{4}}\bigg\|\frac{\omg}{r^{d-2}}\bigg\|_{L^{1}(\bbR^{d})}^{\frac{1}{4}}\bigg\|\frac{\omg}{r^{d-2}}\bigg\|_{L^{\ift}(\bbR^{d})}^{\frac{1}{2}}
	\end{equation} under the assumption that the right hand side is finite. 
\end{prop}

We use the following two lemmas to prove the above proposition. The first one shows the upper bound of the derivative of $ F_{d} $.
 
\begin{lem}\label{lem_Fd'est}
	$ F_{d}' $ satisfies
	\begin{equation}\label{eq_Fdprime}
		|F_{d}'(s)|\lesssim\min\bigg\lbrace\frac{1}{s}, \frac{1}{s^{\frac{d}{2}+1}}\bigg\rbrace,\quad s>0.
	\end{equation}
\end{lem}

The second one is a basic lemma which is frequently used to estimate the convolution $ \frac{1}{|x|}\ast f $ when $ f\in (L^{1}\cap L^{\ift})(\bbR^{2}) $.

\begin{lem}[{{\cite[Lem. 2.10]{FeSv}}}]\label{lem_kernelf}
	Let $ f\in (L^{1}\cap L^{\ift})(\bbR^{2}) $ and let $ K:\bbR^{2}\To\bbR $ satisfy
	$$ |K(x)|\lesssim\frac{1}{|x-x_{0}|},\quad x\in\bbR^{2}, $$
	for some point $ x_{0}\in\bbR^{2} $. 
	Then
	\begin{equation}\label{eq_kernelf}
		\bigg|\int_{\bbR^{2}}K(x)f(x)dx\bigg|\lesssim\|f\|_{L^{1}(\bbR^{2})}^{\frac{1}{2}}\|f\|_{L^{\ift}(\bbR^{2})}^{\frac{1}{2}}.
	\end{equation}
\end{lem}

Assuming that Lemma \ref{lem_Fd'est} holds, we can conclude Proposition \ref{prop_urest}.

\begin{proof}[Proof of Proposition \ref{prop_urest}]
	First, we let $ \lmb>0, z_{0}\in(-\ift,\ift) $ and define
	$$ \tld{u}^{r}(r,z):=u^{r}(\lmb r,\lmb z+z_{0}),\quad \tld{\omg}(r,z):=\lmb\omg(\lmb r,\lmb z+z_{0}). $$
	Then $ \tld{u}^{r} $ and $ \tld{\omg} $ satisfy the same relation between $ u^{r} $ and $ \omg $ as in \eqref{eq_urform}. Now we claim
	\begin{equation}\label{eq_urscaling}
		|u^{r}(1,0)|\lesssim\|r^{d-2}\omg\|_{L^{1}(\bbR^{d})}^{\frac{1}{4}}\bigg\|\frac{\omg}{r^{d-2}}\bigg\|_{L^{1}(\bbR^{d})}^{\frac{1}{4}}\bigg\|\frac{\omg}{r^{d-2}}\bigg\|_{L^{\ift}(\bbR^{d})}^{\frac{1}{2}}.
	\end{equation}
	Once this is shown, then \eqref{eq_urscaling} holds for $ \tld{u}^{r} $ and $ \tld{\omg} $ as well:
	$$ |\tld{u}^{r}(1,0)|\lesssim\|r^{d-2}\tld{\omg}\|_{L^{1}(\bbR^{d})}^{\frac{1}{4}}\bigg\|\frac{\tld{\omg}}{r^{d-2}}\bigg\|_{L^{1}(\bbR^{d})}^{\frac{1}{4}}\bigg\|\frac{\tld{\omg}}{r^{d-2}}\bigg\|_{L^{\ift}(\bbR^{d})}^{\frac{1}{2}}. 
	$$
	Then the left-hand side is
	$|\tld{u}^{r}(1,0)|=|u^{r}(\lmb,z_{0})|,$
	and the right-hand side becomes
	\begin{equation*}
		\begin{split}
			\|r^{d-2}\tld{\omg}\|_{L^{1}(\bbR^{d})}^{\frac{1}{4}}\bigg\|\frac{\tld{\omg}}{r^{d-2}}\bigg\|_{L^{1}(\bbR^{d})}^{\frac{1}{4}}\bigg\|\frac{\tld{\omg}}{r^{d-2}}\bigg\|_{L^{\ift}(\bbR^{d})}^{\frac{1}{2}}&=\lmb^{-\frac{2d-3}{4}}\|r^{d-2}\omg\|_{L^{1}(\bbR^{d})}^{\frac{1}{4}}\cdot\lmb^{-\frac{1}{4}}\bigg\|\frac{\omg}{r^{d-2}}\bigg\|_{L^{1}(\bbR^{d})}^{\frac{1}{4}}\cdot\lmb^{\frac{d-1}{2}}\bigg\|\frac{\omg}{r^{d-2}}\bigg\|_{L^{\ift}(\bbR^{d})}^{\frac{1}{2}}\\
			& =\|r^{d-2}\omg\|_{L^{1}(\bbR^{d})}^{\frac{1}{4}}\bigg\|\frac{\omg}{r^{d-2}}\bigg\|_{L^{1}(\bbR^{d})}^{\frac{1}{4}}\bigg\|\frac{\omg}{r^{d-2}}\bigg\|_{L^{\ift}(\bbR^{d})}^{\frac{1}{2}}.
		\end{split}
	\end{equation*}
	Taking supremum for $ (\lmb,z_{0})\in\Pi $ on the left-hand side finishes the proof.\\
	
	Note that
	\begin{equation*}
		u^{r}(1,0)=-c_{d}\iint_{\Pi}r^{\frac{d}{2}-2}zF_{d}'\bigg(\frac{(r-1)^{2}+z^{2}}{r}\bigg)\omg(r,z)dzdr,
	\end{equation*}
	and
	$$ \|r^{d-2}\omg\|_{L^{1}(\bbR^{d})}^{\frac{1}{4}}\bigg\|\frac{\omg}{r^{d-2}}\bigg\|_{L^{1}(\bbR^{d})}^{\frac{1}{4}}\bigg\|\frac{\omg}{r^{d-2}}\bigg\|_{L^{\ift}(\bbR^{d})}^{\frac{1}{2}}\simeq_{d}\|r^{2(d-2)}\omg\|_{L^{1}(\Pi)}^{\frac{1}{4}}\|\omg\|_{L^{1}(\Pi)}^{\frac{1}{4}}\bigg\|\frac{\omg}{r^{d-2}}\bigg\|_{L^{\ift}(\Pi)}^{\frac{1}{2}}.
	$$
	We claim
	\begin{equation}\label{eq_urintbd}
		\bigg|r^{\frac{d}{2}-2}zF_{d}'\bigg(\frac{(r-1)^{2}+z^{2}}{r}\bigg)\bigg|\lesssim\begin{cases}
			\frac{1}{[(r-1)^{2}+z^{2}]^{\frac{1}{2}}} & \text{ in }I_{1}:=\lbrace(r,z)\in\Pi : r\in[\frac{1}{2},2], z\in[-1,1]\rbrace\\
			\frac{1}{(r-1)^{2}+z^{2}} & \text{ in }I_{2}:=\Pi\setminus I_{1}
		\end{cases}.
	\end{equation}
	Once this is shown, then we get
	\begin{equation*}
		\begin{split}
			\bigg|\iint_{I_{2}}r^{\frac{d}{2}-2}zF_{d}'\bigg(\frac{(r-1)^{2}+z^{2}}{r}\bigg)\omg dzdr\bigg|&\lesssim\iint_{I_{2}}\frac{1}{(r-1)^{2}+z^{2}}\cdot r^{\frac{d}{2}-1}\cdot\omg^{\frac{1}{4}}\cdot\omg^{\frac{1}{4}}\cdot\frac{\omg^{\frac{1}{2}}}{r^{\frac{d}{2}-1}}dzdr\\
			&\lesssim\|r^{2(d-2)}\omg\|_{L^{1}(I_{2})}^{\frac{1}{4}}\|\omg\|_{L^{1}(I_{2})}^{\frac{1}{4}}\bigg\|\frac{\omg}{r^{d-2}}\bigg\|_{L^{\ift}(I_{2})}^{\frac{1}{2}},
		\end{split}
	\end{equation*}
	 and using \eqref{eq_kernelf} from Lemma \ref{lem_kernelf}, we have
	 \begin{equation*}
	 	\begin{split}
	 		\bigg|\iint_{I_{1}}r^{\frac{d}{2}-2}zF_{d}'\bigg(\frac{(r-1)^{2}+z^{2}}{r}\bigg)\omg dzdr\bigg|&\lesssim\|\omg\|_{L^{1}(I_{1})}^{\frac{1}{2}}\|\omg\|_{L^{\ift}(I_{1})}^{\frac{1}{2}}\lesssim\|r^{2(d-2)}\omg\|_{L^{1}(I_{1})}^{\frac{1}{4}}\|\omg\|_{L^{1}(I_{1})}^{\frac{1}{4}}\bigg\|\frac{\omg}{r^{d-2}}\bigg\|_{L^{\ift}(I_{1})}^{\frac{1}{2}}.
	 	\end{split}
	 \end{equation*}
 	Combining these two finishes the proof of \eqref{eq_urscaling}.\\
 	
 	We use \eqref{eq_Fdprime} from Lemma \ref{lem_Fd'est} to prove \eqref{eq_urintbd}. Note that in $ I_{1} $, we have
 	$$ \bigg|r^{\frac{d}{2}-2}zF_{d}'\bigg(\frac{(r-1)^{2}+z^{2}}{r}\bigg)\bigg|\lesssim\frac{r^{\frac{d}{2}-1}|z|}{(r-1)^{2}+z^{2}}\lesssim\frac{1}{[(r-1)^{2}+z^{2}]^{\frac{1}{2}}},
 	$$
 	due to $ r\leq2 $ and $ |z|\leq[(r-1)^{2}+z^{2}]^{\frac{1}{2}} $. Also in $ I_{2} $, we get
 	$$ \bigg|r^{\frac{d}{2}-2}zF_{d}'\bigg(\frac{(r-1)^{2}+z^{2}}{r}\bigg)\bigg|\lesssim\frac{r^{d-1}|z|}{[(r-1)^{2}+z^{2}]^{\frac{d}{2}+1}}\lesssim\frac{1}{(r-1)^{2}+z^{2}},
 	$$
 	because of
 	$$ r\leq[r^{2}+z^{2}]^{\frac{1}{2}}\leq\begin{cases}
 		[(r-1)^{2}+z^{2}]^{\frac{1}{2}} & \text{ in } I_{2}\cap\lbrace r\leq \frac{1}{2}\rbrace, \\
 		2[(r-1)^{2}+z^{2}]^{\frac{1}{2}} & \text{ in } I_{2}\cap\lbrace r\geq 2\rbrace. 
 	\end{cases} $$
\end{proof}

\begin{proof}[Proof of Lemma \ref{lem_Fd'est}]
	For the first part, we prove that there exists $ C>0 $ that satisfies
	$$ s|F_{d}'(s)|\leq C,\quad 0<s<1. $$
	Using the fact
	$$ \sin\tht=\tht+O(\tht^{3})\quad\text{and}\quad 2(1-\cos\tht)=\tht^{2}+O(\tht^{4})\quad\text{as}\quad\tht\To0, $$
	we take a sufficiently small $ \veps_{0}>0 $ such that the following holds:
	$$ \sin\tht\leq 2\tht,\quad 2(1-\cos\tht)\geq\frac{\tht^{2}}{2},\quad \tht\in[0,\veps_{0}]. $$
	We let $ 0<s<1 $. First, we split the upper bound of $ s|F_{d}'(s)| $ into
	\begin{equation*}
		\begin{split}
			s|F_{d}'(s)|&\lesssim\int_{0}^{\pi}\frac{s\sin^{d-3}\tht}{[2(1-\cos\tht)+s]^{\frac{d}{2}}}d\tht=\underbrace{\int_{0}^{\veps_{0}}\frac{s\sin^{d-3}\tht}{[2(1-\cos\tht)+s]^{\frac{d}{2}}}d\tht}_{=(A)}+\underbrace{\int_{\veps_{0}}^{\pi}\frac{s\sin^{d-3}\tht}{[2(1-\cos\tht)+s]^{\frac{d}{2}}}d\tht}_{=(B)}.
		\end{split}
	\end{equation*}
	For $ (A) $, we use the change of variables $ \alp=\frac{\tht}{\sqrt{2s}} $ to get
	\begin{equation*}
		\begin{split}
			(A)&\leq\int_{0}^{\veps_{0}}\frac{s(2\tht)^{d-3}}{(\frac{\tht^{2}}{2}+s)^{\frac{d}{2}}}d\tht=2^{\frac{3d}{2}-3}\int_{0}^{\veps_{0}}\frac{s\tht^{d-3}}{(\tht^{2}+2s)^{\frac{d}{2}}}d\tht=2^{\frac{3d}{2}-4}\int_{0}^{\frac{\veps_{0}}{\sqrt{2s}}}\frac{\alp^{d-3}}{(\alp^{2}+1)^{\frac{d}{2}}}d\alp\leq 2^{\frac{3d}{2}-4}\int_{0}^{\ift}\frac{\alp^{d-3}}{(\alp^{2}+1)^{\frac{d}{2}}}d\alp<\ift.
		\end{split}
	\end{equation*}
	For $ (B) $, we use the fact $ 2(1-\cos\tht)\geq c $ for some $ c>0 $ when $ \tht\in[\veps_{0},\pi] $ to get
	$$ (B)\leq\int_{\veps_{0}}^{\pi}\frac{s}{c^{\frac{d}{2}}}d\tht<\ift.
	$$
	For the second part, we prove
	$$ |F_{d}'(s)|\lesssim\frac{1}{s^{\frac{d}{2}+1}},\quad s\geq M. $$
	for some $ M\geq1 $. We let $ \tau:=\frac{1}{s} $ and define $ g(\tau):=F_{d}'(\frac{1}{\tau}) $. Then we have
	\begin{equation*}
		\begin{split}
			g(\tau)&=-\int_{0}^{\pi}\frac{\cos\tht\sin^{d-3}\tht}{[2(1-\cos\tht)+\frac{1}{\tau}]^{\frac{d}{2}}}d\tht=-\tau^{\frac{d}{2}}\int_{0}^{\pi}\frac{\cos\tht\sin^{d-3}\tht}{[2(1-\cos\tht)\tau+1]^{\frac{d}{2}}}d\tht.
		\end{split}
	\end{equation*}
	Then we use the expansion
	$$ \frac{1}{(x+1)^{\frac{d}{2}}}=1-\frac{d}{2}x+O(x^{2})
	$$
	to get
	\begin{equation*}
		\begin{split}
			g(\tau)&=-\tau^{\frac{d}{2}}\bigg[\int_{0}^{\pi}\cos\tht\sin^{d-3}\tht d\tht-\tau\cdot\frac{d}{2}\int_{0}^{\pi}2\cos\tht\sin^{d-3}\tht(1-\cos\tht)d\tht+O(\tau^{2})\bigg]
			=-C_{d}\tau^{\frac{d}{2}+1}+O(\tau^{\frac{d}{2}+2}),
		\end{split}
	\end{equation*}
	for some constant $ C_{d}>0 $. Thus, 
	we get
	$$ |F_{d}'(s)|=\bigg|g\bigg(\frac{1}{s}\bigg)\bigg|\lesssim\frac{1}{s^{\frac{d}{2}+1}}.
	$$
\end{proof}

\subsection{Proof of Theorem \ref{thm:1}}

Before proving Theorem \ref{thm:1}, we present the following estimates for $ \nrm{\omg}_{L^{\ift}} $ and $ \nrm{r^{2}\omg}_{L^{1}} $. This is an extension of Saint-Raymond \cite[Prop. 2.3]{Raymond} to $ \bbR^{4} $.
 
\begin{lem}\label{lem_omgLiftest}
	Let $ u_{0} $ and $ \omg_{0} $ satisfy the assumptions in Theorem \ref{thm:1}. Then for all $ t\geq0 $, we have
	\begin{equation}\label{eq_omgLiftest}
		\|\omg(t)\|_{L^{\ift}(\bbR^{4})}\leq\bigg(\bigg\|\frac{\omg_{0}}{r^{2}}\bigg\|_{L^{\ift}(\bbR^{4})}+\|\omg_{0}\|_{L^{\ift}(\bbR^{4})}\bigg)\bigg(1+\int_{0}^{t}\|u^{r}(s)\|_{L^{\ift}(\bbR^{4})}ds\bigg)^{2}
	\end{equation}
	and
	\begin{equation}\label{eq_r2omgest}
		\|r^{2}\omg(t)\|_{L^{1}(\bbR^{4})}\leq\bigg(\bigg\|\frac{\omg_{0}}{r^{2}}\bigg\|_{L^{1}(\bbR^{4})}+\|r^{2}\omg_{0}\|_{L^{1}(\bbR^{4})}\bigg)\bigg(1+\int_{0}^{t}\|u^{r}(s)\|_{L^{\ift}(\bbR^{4})}ds\bigg)^{4}.
	\end{equation}
\end{lem}

\begin{proof}[Proof of Lemma \ref{lem_omgLiftest}]
	We fix $ t\geq0 $ and set
	$$ R(t):=1+\int_{0}^{t}\|u^{r}(s)\|_{L^{\ift}(\bbR^{4})}ds, $$
	with $ R_{0}=R(0) $. Then we denote the flow map $ \Phi_{t}(\cdot)=\Phi(t,\cdot) $, which is defined as the unique solution to the ODE
	\begin{equation}\label{eq_flowmapode}
		\frac{d}{dt}\Phi_{t}(x)=u\big(t,\Phi_{t}(x)\big),\quad \Phi_{0}(x)=x.
	\end{equation}
	Note that $ \Phi $ is well-defined because of $ u(t)\in C^{1}(\bbR^{4}) $, by the Sobolev embedding theorem. Then observe that the relative vorticity $ \xi=r^{-2}\omg $ is conserved along the flow $ \Phi $:
	$$ \xi\big(t,\Phi_{t}(x)\big)=\xi_{0}(x), $$
	which implies
	$$ \frac{\omg\big(t,\Phi_{t}(x)\big)}{[\Phi_{t}^{r}(x)]^{2}}=\frac{\omg_{0}(x)}{r_{x}^{2}}, $$
	where $ \Phi_{t}^{r} $ is the $ r- $th component of $ \Phi_{t} $ and $ r_{x}=\sqrt{x_{1}^{2}+x_{2}^{2}+x_{3}^{2}} $.\\
	We denote $ A(t):=\lbrace x\in\bbR^{4} : r_{x}\leq R(t)\rbrace $. To prove \eqref{eq_omgLiftest}, first we consider the case $ x\in A(t) $. Then we have
	\begin{equation*}
		\begin{split}
			|\omg(t,x)|&=r_{x}^{2}\cdot\frac{|\omg(t,x)|}{r_{x}^{2}}\leq R(t)^{2}\bigg\|\frac{\omg(t)}{r^{2}}\bigg\|_{L^{\ift}(\bbR^{4})}=R(t)^{2}\bigg\|\frac{\omg_{0}}{r^{2}}\bigg\|_{L^{\ift}(\bbR^{4})}.
		\end{split}
	\end{equation*}
	Now we consider the other case $ x\in A(t)^{C}=\lbrace x\in\bbR^{4} : r_{x}>R(t)\rbrace $. Here, we take $ y\in\bbR^{4} $ that satisfies $ \Phi_{t}(y)=x $. Then we have the relation
	$$ r_{x}=\Phi_{t}^{r}(y)=r_{y}+\int_{0}^{t}u^{r}\big(s,\Phi_{s}(y)\big)ds, $$
	which gives us $ r_{y}>R(0)=1 $ and
	$$ r_{x}\leq r_{y}+\int_{0}^{t}\|u^{r}(s)\|_{L^{\ift}(\bbR^{4})}ds. $$
	Using these, we have
	$$ \frac{r_{x}}{r_{y}}\leq 1+\frac{1}{r_{y}}\int_{0}^{t}\|u^{r}(s)\|_{L^{\ift}(\bbR^{4})}ds\leq1+\int_{0}^{t}\|u^{r}(s)\|_{L^{\ift}(\bbR^{4})}ds=R(t). $$
	From this, we get
	\begin{equation*}
		\begin{split}
			|\omg(t,x)|&=r_{x}^{2}\cdot\frac{|\omg(t,x)|}{r_{x}^{2}}=r_{x}^{2}\cdot\frac{\big|\omg\big(t,\Phi_{t}(y)\big)\big|}{[\Phi_{t}^{r}(y)]^{2}}=r_{x}^{2}\cdot\frac{|\omg_{0}(y)|}{r_{y}^{2}}\leq R(t)^{2}\|\omg_{0}\|_{L^{\ift}(\bbR^{4})}.
		\end{split}
	\end{equation*}
	Thus, we obtain
	$$ \|\omg(t)\|_{L^{\ift}(\bbR^{4})}\leq\max\bigg\lbrace\bigg\|\frac{\omg_{0}}{r^{2}}\bigg\|_{L^{\ift}(\bbR^{4})},\;\|\omg_{0}\|_{L^{\ift}(\bbR^{4})}\bigg\rbrace R(t)^{2}\leq\bigg(\bigg\|\frac{\omg_{0}}{r^{2}}\bigg\|_{L^{\ift}(\bbR^{4})}+\|\omg_{0}\|_{L^{\ift}(\bbR^{4})}\bigg)R(t)^{2}. $$
	Similarly, to prove \eqref{eq_r2omgest}, we split the domain of $ \|r^{2}\omg(t)\|_{L^{1}(\bbR^{4})} $ into $ A(t) $ and $ A(t)^{C} $:
	$$ \int_{\bbR^{4}}r_{x}^{2}|\omg|dx=\underbrace{\int_{A(t)}r_{x}^{2}|\omg(t,x)|dx}_{=(I)}+\underbrace{\int_{A(t)^{C}}r_{x}^{2}|\omg(t,x)|dx}_{=(II)}. $$
	For $ (I) $, we have
	$$ (I)=\int_{A(t)}r_{x}^{4}\cdot\frac{|\omg(t,x)|}{r_{x}^{2}}dx\leq R(t)^{4}\cdot\int_{A(t)}\frac{|\omg(t,x)|}{r_{x}^{2}}dx\leq R(t)^{4}\bigg\|\frac{\omg(t)}{r^{2}}\bigg\|_{L^{1}(\bbR^{4})}=R(t)^{4}\bigg\|\frac{\omg_{0}}{r^{2}}\bigg\|_{L^{1}(\bbR^{4})}. $$
	For $ (II) $, we get
	\begin{equation*}
		\begin{split}
			(II)&=\int_{A(t)^{C}}r_{x}^{4}\cdot\frac{|\omg(t,x)|}{r_{x}^{2}}dx=\int_{\Phi_{-t}\big(A(t)^{C}\big)}\Phi_{t}^{r}(y)^{4}\cdot\frac{|\omg\big(t,\Phi_{t}(y)\big)|}{[\Phi_{t}^{r}(y)]^{2}}\cdot|\nb_{y}\Phi_{t}(y)|dy\\
			&=\int_{\Phi_{-t}\big(A(t)^{C}\big)}\Phi_{t}^{r}(y)^{4}\cdot\frac{|\omg_{0}(y)|}{r_{y}^{2}}dy\leq\int_{\lbrace r_{y}>1\rbrace}\bigg(\frac{\Phi_{t}^{r}(y)}{r_{y}}\bigg)^{4}\cdot r_{y}^{2}|\omg_{0}(y)|dy\leq R(t)^{4}\|r^{2}\omg_{0}\|_{L^{1}(\bbR^{4})}.
		\end{split}
	\end{equation*}
	Here, $ \Phi_{-t} $ denotes the inverse map of $ \Phi_{t} $. Hence, we have
	$$ \|r^{2}\omg(t)\|_{L^{1}(\bbR^{4})}
	\leq\bigg(\bigg\|\frac{\omg_{0}}{r^{2}}\bigg\|_{L^{1}(\bbR^{4})}+\|r^{2}\omg_{0}\|_{L^{1}(\bbR^{4})}\bigg)R(t)^{4}. $$
\end{proof}

\begin{proof}[Proof of Theorem \ref{thm:1}]
	We fix $ t\geq0 $. Then we use \eqref{eq_urest} from Proposition \ref{prop_urest} and \eqref{eq_r2omgest} from Lemma \ref{lem_omgLiftest} to get
	\begin{equation*}
		\begin{split}
			\dot{R}(t)&=\|u^{r}(t)\|_{L^{\ift}(\bbR^{4})}\leq C\|r^{2}\omg(t)\|_{L^{1}(\bbR^{4})}^{\frac{1}{4}}\bigg\|\frac{\omg(t)}{r^{2}}\bigg\|_{L^{1}(\bbR^{4})}^{\frac{1}{4}}\bigg\|\frac{\omg(t)}{r^{2}}\bigg\|_{L^{\ift}(\bbR^{4})}^{\frac{1}{2}}\\
			&\leq C \bigg(\bigg\|\frac{\omg_{0}}{r^{2}}\bigg\|_{L^{1}(\bbR^{4})}+\|r^{2}\omg_{0}\|_{L^{1}(\bbR^{4})}\bigg)^{\frac{1}{4}}\bigg\|\frac{\omg_{0}}{r^{2}}\bigg\|_{L^{1}(\bbR^{4})}^{\frac{1}{4}}\bigg\|\frac{\omg_{0}}{r^{2}}\bigg\|_{L^{\ift}(\bbR^{4})}^{\frac{1}{2}}R(t),
		\end{split}
	\end{equation*}
	for some $ C>0 $. Denoting the coefficient of $ R(t) $ in the above as $ C_{0}>0 $, we have $ R(t)\leq e^{C_{0}t}.$
	Then from \eqref{eq_omgLiftest} in Lemma \ref{lem_omgLiftest}, we obtain
	\begin{equation*}
		\begin{split}
			\|\omg(t)\|_{L^{\ift}(\bbR^{4})}&\leq C_{1}R(t)^{2} \leq C_{1}e^{2C_{0}t},
		\end{split}
	\end{equation*}
	with $ C_{1}=\|r^{-2}\omg_{0}\|_{L^{\ift}(\bbR^{4})}+\|\omg_{0}\|_{L^{\ift}(\bbR^{4})} $. Hence, the BKM criterion guarantees the global regularity.
\end{proof}

When the initial vorticity is of compact support, the support growth can be only exponential in time. 
\begin{prop}\label{prop_Rtgrowth}
	Let $ S(t):=\sup_{\tau \in [0,t]} \sup\lbrace r : (r,z)\in\supp \omg(\tau)\rbrace $, where $ \omg(t) $ is the solution of \eqref{eq_vortformNd} with a smooth initial data $ \omg_{0} $ with compact support. 
	Then, we have for some $C>0$ depending on the initial data that
	\begin{equation}\label{eq_Rtgrowth}
		S(t)\leq S(0)e^{Ct},\quad t\geq0.
	\end{equation}
\end{prop}

\begin{proof}
	Using \eqref{eq_urest}, we have
	\begin{equation*}
		\begin{split}
			S'(t)&\leq\|u^{r}(t)\|_{L^{\ift}(\bbR^{4})}\lesssim\|r^{2}\omg(t)\|_{L^{1}(\bbR^{4})}^{\frac{1}{4}}\bigg\|\frac{\omg(t)}{r^{2}}\bigg\|_{L^{1}(\bbR^{4})}^{\frac{1}{4}}\bigg\|\frac{\omg(t)}{r^{2}}\bigg\|_{L^{\ift}(\bbR^{4})}^{\frac{1}{2}}\\
			&\leq S(t)\bigg\|\frac{\omg(t)}{r^{2}}\bigg\|_{L^{1}(\bbR^{4})}^{\frac{1}{2}}\bigg\|\frac{\omg(t)}{r^{2}}\bigg\|_{L^{\ift}(\bbR^{4})}^{\frac{1}{2}}=S(t)\bigg\|\frac{\omg_{0}}{r^{2}}\bigg\|_{L^{1}(\bbR^{4})}^{\frac{1}{2}}\bigg\|\frac{\omg_{0}}{r^{2}}\bigg\|_{L^{\ift}(\bbR^{4})}^{\frac{1}{2}}.
		\end{split}
	\end{equation*}
	Solving the ODE gives us \eqref{eq_Rtgrowth}.
\end{proof}

\subsection{Proof of Theorem \ref{thm:2}}

\begin{proof}[Proof of Theorem \ref{thm:2}]
	Using \eqref{eq_urest} and denoting $S(t) = \sup_{\tau \in [0,t]} \sup\lbrace r : (r,z)\in\supp \omg(\tau)\rbrace $, we estimate \begin{equation*}
		\begin{split}
			\nrm{ r^{d-2}{\omg} }_{L^1(\bbR^d)} \lesssim S(t)^{d-3} \nrm{r\omg}_{L^1(\bbR^d)}
		\end{split}
	\end{equation*} to get \begin{equation*}
		\begin{split}
			\nrm{ u^r}_{L^\infty(\bbR^d)} \lesssim S(t)^{\frac{d-3}{4}} \nrm{ r \omg}_{L^{1}(\bbR^d)}^{\frac14}\bigg\|\frac{\omg}{r^{d-2}}\bigg\|_{L^1(\bbR^d)}^{\frac14}  \bigg\|\frac{\omg}{r^{d-2}}\bigg\|_{L^\infty(\bbR^d)}^{\frac12} \lesssim S(t)^{\frac{d-3}{4}},
		\end{split}
	\end{equation*} where we have used that  $\nrm{ r \omg}_{L^{1}(\bbR^d)}$ is conserved in time when $\omg$ is of single sign. Then, \begin{equation*}
	\begin{split}
		S'(t) \le \nrm{u^r}_{L^\infty(\bbR^d)} \lesssim  S(t)^{\frac{d-3}{4}}
	\end{split}
\end{equation*} shows that $S(t) \lesssim (C + t)^{\frac{4}{7-d}}$ for $d \le 6$ and $S(t) \lesssim e^{Ct}$ for $d = 7$. Finally, for $d\le 7$, we obtain that $\omg(t)$ is bounded: \begin{equation*}
\begin{split}
	\nrm{\omg(t)}_{L^\infty(\bbR^d)} \lesssim  \bigg\|\frac{\omg_0}{r^{d-2}}\bigg\|_{L^\infty(\bbR^d)}  S(t)^{d-2},
\end{split}
\end{equation*} which gives global regularity by the BKM criterion. 
\end{proof}

\subsection*{Acknowledgments}
 KC has been supported by the National Research Foundation of Korea (NRF-2018R1D1A1B07043065, 2022R1A4A1032094).
IJ has been supported  by the New Faculty Startup Fund from Seoul National University and the Samsung Science and Technology Foundation under Project Number SSTF-BA2002-04.

\bibliographystyle{plain}
\bibliography{CJL}

\end{document}